\documentclass[british]{article}
\usepackage{ae,aecompl}
\usepackage[T1]{fontenc}
\usepackage[latin9]{inputenc}
\usepackage{amsmath}
\usepackage{amsthm}
\usepackage{amssymb}

\makeatletter
\theoremstyle{plain}
\newtheorem{thm}{\protect\theoremname}
\theoremstyle{definition}
\newtheorem{defn}[thm]{\protect\definitionname}
\theoremstyle{remark}
\newtheorem{rem}[thm]{\protect\remarkname}
\theoremstyle{plain}
\newtheorem{lem}[thm]{\protect\lemmaname}
\theoremstyle{plain}
\newtheorem{prop}[thm]{\protect\propositionname}

\date{}

\makeatother

\usepackage{babel}
\providecommand{\definitionname}{Definition}
\providecommand{\lemmaname}{Lemma}
\providecommand{\propositionname}{Proposition}
\providecommand{\remarkname}{Remark}
\providecommand{\theoremname}{Theorem}

\begin{document}
\title{Local Central Limit Theorem for Multi-Group Curie-Weiss Models}
\author{Michael Fleermann, Werner Kirsch, and Gabor Toth}
\maketitle
\begin{abstract}
We define a multi-group version of the mean-field spin model, also
called Curie-Weiss model. It is known that, in the high temperature
regime of this model, a central limit theorem holds for the vector
of suitably scaled group magnetisations, that is the sum of spins
belonging to each group. In this article, we prove a local central
limit theorem for the group magnetisations in the high temperature
regime.
\end{abstract}


\section{Introduction}

The Curie-Weiss model is a model of ferromagnetism. In its classic
form it is defined through a probability distribution on the set of
spin configurations $\left\{ -1,1\right\} ^{n}$, given by

\[
\mathbb{P}\left(X_{1}=x_{1},\ldots,X_{n}=x_{n}\right)=Z^{-1}\exp\left(\frac{\beta}{2n}\left(\sum_{i=1}^{n}x_{i}\right)\right),
\]

where $Z$ is a normalisation constant that depends on $n$ and $\beta$.
The parameter $\beta$ is called the inverse temperature. It induces
correlation between individual spins, causing spins to align in the
same direction. At low values of $\beta$ (`high temperature'),
the spins are `nearly independent'. At high values of $\beta$ (`low
temperature'), the spins are strongly correlated. There is a critical
value of $\beta=1$, where the collective behaviour of spins changes.
This is called a phase transition. The Curie-Weiss model has been
well studied, and hence the literature is far too extensive to cite
here. The model was first defined by by Husimi \cite{Husimi} and
Temperley \cite{Temperley}. Discussions of it can be found in Kac
\cite{Kac}, Thompson \cite{Thompson}, and Ellis \cite{Ellis}. More
recently, the Curie-Weiss model has been used in the context of social
and political interactions. See e.g. \cite{CGh,Penrose,GBC,KT_CWM,Toth}.
Another area the Curie-Weiss model has found application is the study
of random matrices (see \cite{FriesenLoewe,HKW,KK,FKK,FH,Fleermann}).

In this article, we deal with a multi-group version of this model.
Multi-group versions of the Curie-Weiss model have also been studied
recently. Some references are \cite{CG,FC,Collet,FM,LS,KT1,KT2,KT3}.

Let there be $d\in\mathbb{N}$ groups with $n_{\lambda}$ spins in
group $\lambda\in\{1,\ldots,d\}$, $\sum_{\lambda=1}^{d}n_{\lambda}=n$.
The spin variables are

\[
X=\left(X_{11},X_{12},\ldots,X_{1n_{1}},\ldots,X_{d1},X_{d2},\ldots,X_{dn_{d}}\right)\in\{-1,1\}^{n}.
\]

We assume that each of the $d$ groups converges to a fixed proportion
of the overall population:

\begin{equation}
\alpha_{\lambda}:=\lim_{n\rightarrow\infty}\frac{n_{\lambda}}{n}\text{ and }n_{\lambda}\rightarrow\infty\text{ as }n\rightarrow\infty,\label{eq:alpha}
\end{equation}

so that the $\alpha_{\lambda}$ sum to 1.

Instead of a single inverse temperature parameter, there is a coupling
matrix that describes the spin interactions. We will call this matrix

\[
J:=(J_{\lambda\mu})_{\lambda,\mu=1,\ldots,d}.
\]

Every spin in group $\lambda$ interacts with every spin in group
$\mu$ with a strength given by the coupling constant $J_{\lambda\mu}$.

Just as in the single-group model, there is a Hamiltonian function
that assigns to each spin configuration a certain energy level:

\begin{align}
\mathbb{H}(x_{11},\ldots,x_{dn_{d}}):=-\frac{1}{2n}\sum_{\lambda,\mu=1}^{d}J_{\lambda\mu}\sum_{i=1}^{n_{\lambda}}\sum_{j=1}^{n_{\mu}}x_{\lambda i}x_{\mu j}.
\end{align}

As we can see from the definition of $\mathbb{H}$, it suffices to
consider symmetric $J$, for otherwise we can replace $J$ by $\frac{J+J^{T}}{2}$,
leaving the Hamiltonian unchanged. 
\begin{defn}
The Curie-Weiss measure $\mathbb{P}$, which gives the probability
of each of the $2^{n}$ spin configurations, is defined by
\begin{align}
\mathbb{P}\left(X_{11}=x_{11},\ldots,X_{dn_{d}}=x_{dn_{d}}\right):=Z^{-1}\exp\left(-\mathbb{H}\left(x_{11},\ldots,x_{dn_{d}}\right)\right)\label{eq:Pbeta}
\end{align}
where each $x_{\lambda i}\in\{-1,1\}$ and $Z$ is a normalisation
constant which depends on $n$ and $J$.
\end{defn}

We distinguish two different classes of coupling matrices:
\begin{enumerate}
\item Homogeneous coupling matrices\\
\[
J=(\beta)_{\lambda,\mu=1,\ldots,d},
\]
where all entries are equal to the same constant $\beta\geq0$.
\item Heterogeneous coupling matrices\\
\[
J=(J_{\lambda,\mu})_{\lambda,\mu=1,\ldots,d},
\]
which we assume to be positive definite.
\end{enumerate}
\begin{rem}
Without the assumption of positive definiteness, the high temperature
regime (see below) may be empty. For more details, see \cite{KT3}.
\end{rem}

This model has three regimes: The high temperature, the critical,
and the low temperature regime. In each regime, the spins behave differently
and the limiting distribution for large $n$ is different in each
case.

For each group $\lambda$, we define $S_{\lambda}:=\sum_{i=1}^{n_{\lambda}}X_{\lambda i}$
to be the sum of all spins belonging to that group. In this article,
we show a local limit theorem for the normalised magnetisation vector

\[
\left(\frac{S_{1}}{\sqrt{n_{1}}},\cdots,\frac{S_{d}}{\sqrt{n_{d}}}\right)
\]

in the high temperature regime.

If the coupling matrix is homogeneous, then the high temperature regime
is characterised by $\beta<1$.

For heterogeneous coupling matrices, the situation is somewhat more
complicated. Here, the parameter space is

\begin{align*}
\Phi & :=\left\{ (\alpha_{1},\ldots,\alpha_{d})|\alpha_{1},\ldots,\alpha_{d}\geq0,\sum_{\lambda=1}^{d}\alpha_{\lambda}=1\right\} \\
 & \quad\times\left\{ J|J\text{ is a }d\times d\text{ positive definite matrix}\right\} ,
\end{align*}

containing all possible combinations of asymptotic relative group
sizes $(\alpha_{1},\ldots,\alpha_{d})$ as in (\ref{eq:alpha}) and
coupling matrices.

We define

\[
\boldsymbol{\alpha}:=\text{diag}(\alpha_{1},\ldots,\alpha_{d}),
\]

where `diag' stands for a diagonal matrix with the entries given
between parentheses, and

\begin{equation}
H:=J^{-1}-\boldsymbol{\alpha}.\label{eq:def_Hessian}
\end{equation}

Note that this definition of a multi-group Curie-Weiss model reduces
to the classical single-group model if we set $d=1$, since then $n_{1}=n$
and $J=\beta$. See also Remark \ref{rem:d=00003D1}.

The parameter space $\Phi$ is partitioned into three regimes (for
details, see \cite{KT3}). We are only concerned with the high temperature
regime.
\begin{defn}
The `high temperature regime' for heterogeneous coupling matrices
is the set of parameters

\[
\Phi_{h}:=\{\phi\in\Phi|H\text{ is positive definite}\}.
\]
\end{defn}

In the high temperature regime, a multivariate central limit theorem
holds for the normalised sums of spins in each group. For a proof
see e.g. \cite{KT3}.
\begin{thm}
\label{thm:CLT}In the high temperature regime, we have

\[
\left(\frac{S_{1}}{\sqrt{n_{1}}},\ldots,\frac{S_{d}}{\sqrt{n_{d}}}\right)\underset{n\to\infty}{\Longrightarrow}\mathcal{N}((0,\ldots,0),C),
\]

where $\mathcal{N}((0,\ldots,0),C)$ is a zero-mean multivariate normal
distribution with positive definite covariance matrix $C$, and `$\Longrightarrow$'
stands for weak convergence.
\end{thm}

\begin{rem}
In the central limit theorem above, the limiting distribution has
the covariance matrix

\[
C=I+\sqrt{\boldsymbol{\alpha}}\Sigma\sqrt{\boldsymbol{\alpha}}.
\]

The matrix $\Sigma$ depends on the class of coupling matrices:

\[
\Sigma=\begin{cases}
\left(\frac{\beta}{1-\beta}\right)_{\lambda,\mu=1,\ldots,d}, & J\text{ homogeneous,}\\
H^{-1}, & J\text{ heterogeneous.}
\end{cases}
\]
\end{rem}

We shall write $\phi_{C}$ for the density function of $\mathcal{N}((0,\ldots,0),C)$,
and we set

\[
\boldsymbol{S}^{n}:=\left(\frac{S_{1}}{\sqrt{n_{1}}},\ldots,\frac{S_{d}}{\sqrt{n_{d}}}\right).
\]

For a given $n\in\mathbb{N}$ and group $\lambda$, $\frac{S_{\lambda}}{\sqrt{n_{\lambda}}}$
takes values on the grid $\frac{n_{\lambda}+2\mathbb{Z}}{\sqrt{n_{\lambda}}}$.
Hence, the vector $\left(\frac{S_{1}}{\sqrt{n_{1}}},\ldots,\frac{S_{d}}{\sqrt{n_{d}}}\right)$
takes values on the grid

\[
\mathcal{L}_{n}:=\prod_{\lambda=1}^{d}\frac{n_{\lambda}+2\mathbb{Z}}{\sqrt{n_{\lambda}}}.
\]

We show that the central limit theorem above can be strengthened to
a multivariate local limit theorem:
\begin{thm}
\label{thm:LL}In the high temperature regime, the following local
limit theorem holds:

\[
\sup_{x\in\mathcal{L}_{n}}\left|\frac{\prod_{\lambda=1}^{d}\sqrt{n_{\lambda}}}{2^{d}}\mathbb{P}\left(\boldsymbol{S}^{n}=x\right)-\phi_{C}(x)\right|\underset{n\to\infty}{\longrightarrow}0.
\]
\end{thm}

\section{Proof}

We first state two auxiliary lemmas.
\begin{lem}
\label{lem:char_fn}Let $Y:=\left(Y_{1},\ldots,Y_{d}\right)$ be a
random vector on the grid $\prod_{\lambda=1}^{d}\left(v_{\lambda}+w_{\lambda}\mathbb{Z}\right)$
with a characteristic function $\varphi$, defined by $\varphi(t):=\mathbb{E}\left(\exp\left(it\cdot Y\right)\right),t\in\mathbb{R}^{d}$.
The following two properties hold:
\begin{enumerate}
\item $\varphi$ is periodic, i.e. for all $t\in\mathbb{R}^{d},k_{1},\ldots,k_{d}\in\mathbb{Z}$,\\
\\
\[
\varphi\left(t+2\pi\left(\frac{k_{1}}{w_{1}},\cdots,\frac{k_{d}}{w_{d}}\right)\right)=\varphi(t)
\]
\item We have for all $k_{1},\ldots,k_{d}\in\mathbb{Z}$,\\
\\
\[
\left|\varphi\left(2\pi\left(\frac{k_{1}}{w_{1}},\cdots,\frac{k_{d}}{w_{d}}\right)\right)\right|=1
\]
\\
and\\
\\
\[
\left|\varphi\left(t\right)\right|<1
\]
\\
for all $0<t<2\pi\left(\frac{k_{1}}{w_{1}},\cdots,\frac{k_{d}}{w_{d}}\right)$
componentwise.
\end{enumerate}
\end{lem}

\begin{proof}
This follows from a straightforward modification of the proof Theorem
3.5.2 on page 140 in \cite{Durrett}.
\end{proof}
The second statement above gives us an upper bound for the characteristic
function of a random variable on the grid, which we shall use in our
calculations later on.

We have the inversion formula, by which we can recover a discrete
distribution from its characteristic function:
\begin{lem}
Let $\left(Y_{1},\ldots,Y_{d}\right)$ be a random vector as in the
lemma above. Then for all $x\in\prod_{\lambda=1}^{d}\left(v_{\lambda}+w_{\lambda}\mathbb{Z}\right)$,
we have

\[
\mathbb{P}\left(\left(Y_{1},\ldots,Y_{d}\right)=x\right)=\frac{\prod_{\lambda=1}^{d}w_{\lambda}}{\left(2\pi\right)^{d}}\int_{\prod_{\lambda}\left[-\frac{\pi}{w_{\lambda}},\frac{\pi}{w_{\lambda}}\right]}e^{-it\cdot x}\varphi(t)\textup{d}t.
\]
\end{lem}

For a proof see e.g. section 3.10 in \cite{Durrett}.

For integrable characteristic functions $\varphi$, we also have an
inversion formula:
\begin{lem}
Let $\varphi$ be the characteristic function of some $d$-dimensional
random vector with density function $f$, and assume $\varphi$ is
integrable. Then the inversion formula

\[
f(x)=\frac{1}{\left(2\pi\right)^{d}}\int_{\mathbb{R}^{d}}e^{-it\cdot x}\varphi(t)\textup{d}t
\]

allows us to recover the density function $f$.
\end{lem}

We use this inversion formula to show Theorem \ref{thm:LL}. Let $\varphi_{\boldsymbol{S}^{n}}$
be the characteristic function of $\boldsymbol{S}^{n}$ and $\varphi_{\mathcal{N}(C)}$
that of $\mathcal{N}((0,\ldots,0),C)$. We use the symbol $\mathbb{E}$
as the expectation with respect to the probability measure $\mathbb{P}$
of the underlying probability space.

By the lemma, we have

\[
\frac{\prod_{\lambda=1}^{d}\sqrt{n_{\lambda}}}{2^{d}}\mathbb{P}\left(\boldsymbol{S}^{n}=x\right)=\frac{1}{\left(2\pi\right)^{d}}\int_{\prod_{\lambda}\left[-\frac{\pi\sqrt{n_{\lambda}}}{2},\frac{\pi\sqrt{n_{\lambda}}}{2}\right]}e^{-it\cdot x}\varphi_{\boldsymbol{S}^{n}}(t)\textup{d}t,
\]

and, therefore,

\begin{align}
 & \left|\frac{\prod_{\lambda=1}^{d}\sqrt{n_{\lambda}}}{2^{d}}\mathbb{P}\left(\boldsymbol{S}^{n}=x\right)-\phi_{C}(x)\right|\nonumber \\
= & \frac{1}{\left(2\pi\right)^{d}}\left|\int_{\prod_{\lambda}\left[-\frac{\pi\sqrt{n_{\lambda}}}{2},\frac{\pi\sqrt{n_{\lambda}}}{2}\right]}e^{-it\cdot x}\varphi_{\boldsymbol{S}^{n}}(t)\textup{d}t-\int_{\mathbb{R}^{d}}e^{-it\cdot x}\varphi_{\mathcal{N}(C)}(t)\textup{d}t\right|\nonumber \\
\leq & \frac{1}{\left(2\pi\right)^{d}}\int_{\mathbb{R}^{d}}I_{\prod_{\lambda}\left[-\frac{\pi\sqrt{n_{\lambda}}}{2},\frac{\pi\sqrt{n_{\lambda}}}{2}\right]}(t)\left|\varphi_{\boldsymbol{S}^{n}}(t)-\varphi_{\mathcal{N}(C)}(t)\right|\textup{d}t\label{eq:centre}\\
 & +\frac{1}{\left(2\pi\right)^{d}}\int_{\mathbb{R}^{d}\backslash\prod_{\lambda}\left[-\frac{\pi\sqrt{n_{\lambda}}}{2},\frac{\pi\sqrt{n_{\lambda}}}{2}\right]}\left|\varphi_{\mathcal{N}(C)}(t)\right|\textup{d}t.\label{eq:tails}
\end{align}

We see that the term (\ref{eq:tails}) converges to 0 as $n\rightarrow\infty$,
since $\left|\varphi_{\mathcal{N}(C)}(t)\right|$ is integrable. We
also note that the expression (\ref{eq:centre}) is independent of
the point $x\in\mathcal{L}_{n}$. If we can show that (\ref{eq:centre})
converges to 0, then we are done. To this end, we see by Theorem \ref{thm:CLT}
that $\varphi_{S}^{n}(t)\to\varphi_{\mathcal{N}(C)}(t)$ pointwise.
Thus, to show that (\ref{eq:centre}) converges to zero is for the
most part a matter of finding an appropriate integrable majorant,
so that the theorem of dominated convergence can be applied. To construct
a suitable majorant, we need to apply some properties of the multivariate
Curie-Weiss distribution.

Let the Rademacher distribution with parameter $m\in\mathbb{R}$ $\mathcal{R}_{m}$
be defined on $\{\pm1\}$ by the probability of the event $\left\{ 1\right\} $
equal to $\frac{1+\bar{m}}{2}$ setting $\bar{m}:=\tanh m$.

We use the de Finetti representation of the Curie-Weiss measure (see
\cite{KT3}):
\begin{prop}
\label{prop:asymptotic_conc}The distribution of the multi-group Curie-Weiss
model has the following representation: For any spin configuration
$\left(x_{11},\ldots,x_{dn_{d}}\right)$, we have

\[
\mathbb{P}\left(X_{11}=x_{11},\ldots,X_{dn_{s}}=x_{dn_{d}}\right)=\int_{\mathbb{R}^{d}}P_{m}\left(X_{11}=x_{11},\ldots,X_{dn_{s}}=x_{dn_{d}}\right)\mu_{J,n}(\textup{d}m),
\]

where $P_{m}$ is the product measure of Rademacher distributions
with parameters $m_{\lambda}$ for all spins belonging to group $\lambda$.

The de Finetti measure $\mu_{J,n}$ is defined by the density function

\[
f_{J,n}(m)\propto\exp\left(-n\left(\frac{1}{2}m^{T}J^{-1}m-\sum_{\lambda}\frac{n_{\lambda}}{n}\ln\cosh m_{\lambda}\right)\right),m\in\mathbb{R}^{d}.
\]

In the high temperature regime, $\mu_{J,n}$ has an asymptotic concentration
property, such that for all $\delta>0$ there is a $D>0$ with the
property

\[
\mu_{J,n}\left(\mathbb{R}^{d}\backslash\left[-\delta,\delta\right]^{d}\right)<\exp\left(-Dn\right)
\]

for large enough $n$.
\end{prop}

\begin{rem}
\label{rem:d=00003D1}If we set $d=1$, then the de Finetti density
above becomes proportional to

\[
\exp\left(-n\left(\frac{1}{2\beta}m^{2}-\ln\cosh m\right)\right),m\in\mathbb{R}.
\]

From this expression, we obtain the usual de Finetti density defined
on $\left[-1,1\right]$ proportional to

\[
\frac{\exp\left(-\frac{n}{2}\left(\frac{1}{\beta}\left(\frac{1}{2}\ln\frac{1+t}{1-t}\right)+\ln\left(1-t^{2}\right)\right)\right)}{1-t^{2}},t\in\left[-1,1\right]
\]

by the substitution $t:=\tanh m$. Cf. \cite{MM}.
\end{rem}

Let $\varphi_{\mathcal{R}(m)}$ be the characteristic function of
a centred Rademacher distribution with parameter $m$, i.e. the distribution
of $X-\tanh m$ for a random variable $X\sim\mathcal{R}_{m}$, and
let $\textup{E}_{m}$ be the expectation under that distribution.

Now we deal with expression (\ref{eq:centre}). We pick some $0<\delta<\pi/2$
and partition the set $\prod_{\lambda}\left[-\frac{\pi\sqrt{n_{\lambda}}}{2},\frac{\pi\sqrt{n_{\lambda}}}{2}\right]=\prod_{\lambda}\left[-\delta\sqrt{n_{\lambda}},\delta\sqrt{n_{\lambda}}\right]\cup B_{n}=:A_{n}\cup B_{n}$
for each $n\in\mathbb{N}$.

Our goal is to show that (\ref{eq:centre}) converges to 0. We do
so by showing the result separately over $A_{n}$ and $B_{n}$.

The following upper bound holds over $A_{n}$:

\begin{equation}
I_{A_{n}}(t)\left|\varphi_{\boldsymbol{S}^{n}}(t)\right|\leq I_{A_{n}}(t)\int_{\mathbb{R}^{d}}\prod_{\lambda}\left|\varphi_{\mathcal{R}\left(m_{\lambda}\right)}\left(\frac{t}{\sqrt{n_{\lambda}}}\right)\right|^{n_{\lambda}}\mu_{J,n}(\textup{d}m).\label{eq:upper_bound_int}
\end{equation}

We calculate an upper bound for the Rademacher characteristic function:

\begin{align*}
\left|\varphi_{\mathcal{R}\left(m\right)}(u)\right| & =\left|\textup{E}_{m}\exp\left(iu\left(X_{\lambda1}-\bar{m}\right)\right)\right|\\
 & \leq\left|1-\left(1-\bar{m}^{2}\right)\frac{u^{2}}{2}\right|+u^{2}\left(1-\bar{m}^{2}\right)\min\left\{ \left|u\right|\left(1+\bar{m}^{2}\right),1\right\} \\
 & \leq1-\left(1-\bar{m}^{2}\right)\frac{u^{2}}{2}+\left(1-\bar{m}^{2}\right)\frac{u^{2}}{4}\\
 & \leq\exp\left(-\left(1-\bar{m}^{2}\right)\frac{u^{2}}{4}\right).
\end{align*}

The first inequality follows from a Taylor expansion of the exponential
function with the remainder term of order three $u^{2}\textup{E}_{m}\min\left\{ \left|u\right|\left|X_{\lambda1}-\bar{m}\right|^{3},\left|X_{\lambda1}-\bar{m}\right|^{2}\right\} $,
which is smaller or equal $u^{2}\left(1-\bar{m}^{2}\right)\min\left\{ \left|u\right|\left(1+\bar{m}^{2}\right),1\right\} $
as can be verified by direct calculation. The second inequality holds
for small enough $\left|u\right|$. The third inequality for any $u\in\mathbb{R}$
is well-known. Therefore,

\[
\left|\varphi_{\mathcal{R}\left(m_{\lambda}\right)}\left(\frac{t}{\sqrt{n_{\lambda}}}\right)\right|\leq\exp\left(-\left(1-\bar{m}_{\lambda}^{2}\right)\frac{t_{\lambda}^{2}}{4n_{\lambda}}\right),
\]

and we pick some $\tau\in\left(0,1\right)$ to continue with our calculation:

\begin{align}
\eqref{eq:upper_bound_int} & \leq I_{A_{n}}(t)\int_{\mathbb{R}^{d}}\exp\left(-\frac{1}{4}\sum_{\lambda}\left(1-\bar{m}_{\lambda}^{2}\right)t_{\lambda}^{2}\right)\mu_{J,n}(\textup{d}m)\nonumber \\
 & \leq\int_{\left[-\tau,\tau\right]^{d}}\exp\left(-\frac{1}{4}\sum_{\lambda}\left(1-\bar{m}_{\lambda}^{2}\right)t_{\lambda}^{2}\right)\mu_{J,n}(\textup{d}m)\nonumber \\
 & \quad+I_{A_{n}}(t)\mu_{J,n}\left(\mathbb{R}^{d}\backslash\left[-\tau,\tau\right]^{d}\right)\nonumber \\
 & \leq\exp\left(-\frac{1}{4}\left(1-\tanh^{2}\tau\right)\sum_{\lambda}t_{\lambda}^{2}\right)+I_{A_{n}}(t)\exp\left(-\eta n\right),\label{eq:upper_bound_int_2}
\end{align}

where the second term in the last line follows from Lemma \ref{prop:asymptotic_conc}.
Note that $\eta>0$.

It is clear that the first summand in (\ref{eq:upper_bound_int_2})
is integrable. For the second summand, we have

\begin{align*}
I_{A_{n}}(t)\exp\left(-\eta n\right) & \leq I_{A_{1}}\left(t\right)\exp\left(-\eta\right)+\sum_{k=1}^{\infty}I_{A_{k+1}\backslash A_{k}}\left(t\right)\exp\left(-\eta\left(k+1\right)\right)=:f(t).
\end{align*}

Let $\boldsymbol{\lambda}^{d}$ be the Lebesgue measure on $\mathbb{R}^{d}$.
We show that the function $f$ on the right hand side is an integrable
majorant for all $I_{A_{n}}(t)\exp\left(-\eta n\right),n\in\mathbb{N}$:

\begin{align*}
\int_{\mathbb{R}^{d}}f(t)\textup{d}t & =\boldsymbol{\lambda}^{d}\left(A_{1}\right)\exp\left(-\eta\right)+\sum_{k=1}^{\infty}\boldsymbol{\lambda}^{d}\left(A_{k+1}\backslash A_{k}\right)\exp\left(-\eta\left(k+1\right)\right).
\end{align*}

Each summand in the series above can be bounded above by

\begin{align*}
\boldsymbol{\lambda}^{d}\left(A_{k+1}\backslash A_{k}\right)\exp\left(-\eta\left(k+1\right)\right) & \leq\boldsymbol{\lambda}^{d}\left(A_{k+1}\right)\exp\left(-\eta\left(k+1\right)\right)\\
 & =\boldsymbol{\lambda}^{d}\left(\prod_{\lambda}\left[-\delta\sqrt{\left(k+1\right)_{\lambda}},\delta\sqrt{\left(k+1\right)_{\lambda}}\right]\right)\\
 & \leq\left(2\delta\right)^{d}\left(k+1\right)^{\frac{d}{2}}\exp\left(-\eta\left(k+1\right)\right),
\end{align*}

which is summable in $k$.

As $\left|\varphi_{\mathcal{N}(C)}(t)\right|$ is integrable as well,
we have found that $I_{A_{n}}(t)\left|\varphi_{\boldsymbol{S}^{n}}(t)-\varphi_{\mathcal{N}(C)}(t)\right|$
has an integrable majorant. From the global central limit theorem
\ref{thm:CLT}, we know that $\left|\varphi_{\boldsymbol{S}^{n}}(t)-\varphi_{\mathcal{N}(C)}(t)\right|\rightarrow0$
pointwise as $n\rightarrow\infty$, so we conclude that the integral
of $I_{A_{n}}(t)\left|\varphi_{\boldsymbol{S}^{n}}(t)-\varphi_{\mathcal{N}(C)}(t)\right|$
over $\mathbb{R}^{d}$ converges to 0 as $n\rightarrow\infty$.

We proceed with the integrand over the set $B_{n}$:

\begin{align}
I_{B_{n}}(t)\left|\varphi_{\boldsymbol{S}^{n}}(t)\right| & \leq I_{B_{n}}(t)\int_{\mathbb{R}^{d}}\prod_{\lambda}\left|\varphi_{\mathcal{R}\left(m_{\lambda}\right)}\left(\frac{t}{\sqrt{n_{\lambda}}}\right)\right|^{n_{\lambda}}\mu_{J,n}(\textup{d}m)\nonumber \\
 & \leq I_{B_{n}}(t)\int_{\mathbb{R}^{d}}\left(\theta(m)\right)^{n}\mu_{J,n}(\textup{d}m),\label{eq:upper_bound_int_3}
\end{align}

where the existence of

\[
\theta(m)=\max_{t\in B_{n},\lambda=1,\ldots,d}\left|\varphi_{\mathcal{R}\left(m_{\lambda}\right)}\left(t_{\lambda}\right)\right|<1
\]

is a consequence of Lemma \ref{lem:char_fn}. We continue with the
calculation of an upper bound:

\[
\eqref{eq:upper_bound_int_3}\leq I_{B_{n}}(t)\int_{\left[-\tau,\tau\right]^{d}}\left(\theta(m)\right)^{n}\mu_{J,n}(\textup{d}m)+I_{B_{n}}(t)\mu_{J,n}\left(\mathbb{R}^{d}\backslash\left[-\tau,\tau\right]^{d}\right).
\]

On the interval $\left[-\tau,\tau\right]$, $\theta$ is bounded away
from 1:

\[
s:=\sup_{m\in\left[-\tau,\tau\right]}\theta(m)<1.
\]

With this final upper bound for $I_{B_{n}}(t)\left|\varphi_{\boldsymbol{S}^{n}}(t)\right|$,
we see that

\begin{align*}
\int_{\mathbb{R}^{d}}I_{B_{n}}(t)\left|\varphi_{\boldsymbol{S}^{n}}(t)\right|\textup{d}t & \leq\int_{\mathbb{R}^{d}}I_{B_{n}}(t)\left(s^{n}+\mu_{J,n}\left(\mathbb{R}^{d}\backslash\left[-\tau,\tau\right]^{d}\right)\right)\textup{d}t\\
 & \leq2^{d}\left(\frac{\pi}{2}-\delta\right)^{d}\prod_{\lambda}\sqrt{n_{\lambda}}\left(s^{n}+\mu_{J,n}\left(\mathbb{R}^{d}\backslash\left[-\tau,\tau\right]^{d}\right)\right)\\
 & \rightarrow0
\end{align*}

as $n\rightarrow\infty$, due to the exponential decay of the terms
$s^{n}$ and $\mu_{J,n}\left(\mathbb{R}^{d}\backslash\left[-\tau,\tau\right]^{d}\right)$.

Contact information of the authors:

michael.fleermann@fernuni-hagen.de

werner.kirsch@fernuni-hagen.de

gabor.toth@fernuni-hagen.de

\begin{thebibliography}{10}
\bibitem{Collet}Collet, Francesca: Macroscopic Limit of a Bipartite
Curie-Weiss Model: A Dynamical Approach, J. Stat. Phys., 157(6), pp.
1301-1319 (2014)

\bibitem{CG} Contucci, Pierluigi, Gallo, Ignacio: Bipartite Mean
Field Spin Systems. Existence and Solution, Math. Phys. Elec. Jou.
Vol 14, N.1, 1-22 (2008)

\bibitem{CGh}Contucci, Pierluigi and Ghirlanda, S.: Modelling Society
with Statistical Mechanics: an Application to Cultural Contact and
Immigration. Quality and Quantity, 41, 569-578 (2007)

\bibitem{Durrett}Durrett, Rick: Probability Theory and Examples,
Fifth Edition, Cambridge University Press (2019)

\bibitem{Ellis} Ellis, Richard: Entropy, large deviations, and statistical
mechanics, Whiley 1985

\bibitem{FM} Fedele, Micaela: Rescaled Magnetization for Critical
Bipartite Mean-Fields Models, J. Stat. Phys. 155:223--226 (2014)

\bibitem{FC} Fedele, Micaela; Contucci, Pierluigi: Scaling Limits
for Multi-species Statistical Mechanics Mean-Field Models, J. Stat.
Phys. 144:1186--1205 (2011)

\bibitem{Fleermann}Fleermann, Michael: Global and Local Semicircle
Law for Random Matrices with Correlated Entries. PhD thesis. FernUniversität
in Hagen, Germany, 2019.

\bibitem{FH}Fleermann, Michael and Heiny, Johannes: High-dimensional
sample covariance matrices with Curie-Weiss entries. Oct. 2019. URL:
https: https://arxiv.org/pdf/1910.12332.pdf.

\bibitem{FKK}Fleermann, Michael; Kirsch, Werner; Kriecherbauer, Thomas.
The almost sure semicircle law for random band matrices with dependent
entries. Stochastic Processes and their Applications 131 (2021), pp.
172-200.

\bibitem{FriesenLoewe}Friesen, Olga and Löwe, Matthias: A phase transition
for the limiting spectral density of random matrices. Electronic Journal
of Probability 18.17 (2013), pp. 1-17.

\bibitem{GBC}Gallo, Ignacio; Barra, Adriano; Contucci, Pierluigi;
Parameter Evaluation of a Simple Mean-Field Model of Social Interaction,
Math. Models Methods Appl. Sci., 19 (suppl.), pp. 1427-1439 (2009)

\bibitem{HKW}Hochstättler, Winfried; Kirsch, Werner; Warzel, Simone:
Semicircle law for a matrix ensemble with dependent entries. J. Theoret.
Probab. 29.3 (2016), pp. 1047-1068.

\bibitem{Husimi}Husimi, K.: Statistical Mechanics of Condensation,
Proceedings of the International Conference of Theoretical Physics,
pp. 531-533, Science Council of Japan, Tokyo (1953)

\bibitem{Kac}Kac, M.: Mathematical Mechanisms of Phase Transitions,
in Statistical Physics: Phase Transitions and Superfluidity, Vol.
1, pp. 241-305, Brandeis University Summer Institute in Theoretical
Physics (1968)

\bibitem{MM}Kirsch, Werner: A Survey on the Method of Moments, available
from http://www.fernuni-hagen.de/stochastik/

\bibitem{Penrose}Kirsch, Werner: On Penrose's Square-root Law and
Beyond, Homo Oeconomicus 24(3/4): 357--380, 2007

\bibitem{KK}Kirsch, Werner and Kriecherbauer, Thomas: Semicircle
law for generalized Curie-Weiss matrix ensembles at subcritical temperature.
J. Theoret. Probab. 31.4 (2018), pp. 2446-2458.

\bibitem{KT3}Kirsch, Werner; Toth, Gabor: CWM article in preparation

\bibitem{KT1}Kirsch, Werner; Toth, Gabor: Two Groups in a Curie-Weiss
Model. Math Phys Anal Geom 23, 17 (2020)

\bibitem{KT2}Kirsch, Werner; Toth, Gabor: Two Groups in a Curie--Weiss
Model with Heterogeneous Coupling. J Theor Probab 33, 2001--2026
(2020)

\bibitem{KT_CWM}Kirsch, Werner; Toth, Gabor: CWM optimal weights
article

\bibitem{LS} Löwe, Matthias; Schubert, Kristina: Fluctuations for
block spin Ising models. Electron. Commun. Probab. 23 (2018) 

\bibitem{Temperley}Temperley, H.N.V.: The Mayer Theory of Condensation
Tested against a Simple Model of the Imperfect Gas, Proc. Phys. Soc.,
A 67, pp. 233-238 (1954)

\bibitem{Thompson}Thompson, C.J.: Mathematical Statistical Mechanics,
Macmillan (1972)

\bibitem{Toth}Toth, Gabor: Correlated Voting in Multipopulation Models,
Two-Tier Voting Systems, and the Democracy Deficit, PhD Thesis, Fernuniversität
in Hagen. (2020) https://ub-deposit.fernuni-hagen.de/receive/mir\_mods\_00001617
\end{thebibliography}
\end{document}